\pdfoutput=1
\RequirePackage{ifpdf}
\ifpdf % We are running pdfTeX in pdf mode
\documentclass[pdftex]{sigma}
\else
\documentclass{sigma}
\fi

\numberwithin{equation}{section}

\newtheorem{Theorem}{Theorem}[section]
\newtheorem{Corollary}[Theorem]{Corollary}
\newtheorem{Lemma}[Theorem]{Lemma}
\newtheorem{Proposition}[Theorem]{Proposition}
 { \theoremstyle{definition}
\newtheorem{Example}[Theorem]{Example}
\newtheorem{Remark}[Theorem]{Remark} }

\begin{document}

\allowdisplaybreaks

\newcommand{\arXivNumber}{1611.05674}

\renewcommand{\PaperNumber}{027}

\FirstPageHeading

\ShortArticleName{Hopf Algebras which Factorize through the Taft Algebra $T_{m^{2}}(q)$}

\ArticleName{Hopf Algebras which Factorize through the Taft\\ Algebra $\boldsymbol{T_{m^{2}}(q)}$ and the Group Hopf Algebra $\boldsymbol{K[C_{n}]}$}

\Author{Ana-Loredana AGORE~$^{\dag\ddag}$}

\AuthorNameForHeading{A.-L.~Agore}

\Address{$^\dag$~Faculty of Engineering, Vrije Universiteit Brussel, Pleinlaan 2, B-1050 Brussels, Belgium}
\EmailD{\href{mailto:ana.agore@vub.be}{ana.agore@vub.be}, \href{mailto:ana.agore@gmail}{ana.agore@gmail}}
\URLaddressD{\url{http://homepages.vub.ac.be/~aagore/welcome.html}}
\Address{$^\ddag$~``Simion Stoilow'' Institute of Mathematics of the Romanian Academy,\\
\hphantom{$^\ddag$}~P.O.~Box 1-764, 014700 Bucharest, Romania}

\ArticleDates{Received August 28, 2017, in final form March 20, 2018; Published online March 25, 2018}

\Abstract{We completely describe by generators and relations and classify all Hopf algebras which factorize through the Taft algebra $T_{m^{2}}(q)$ and the group Hopf algebra $K[C_{n}]$: they are $nm^{2}$-dimensional quantum groups $T_{nm^{2}}^ {\omega}(q)$ associated to an $n$-th root of unity~$\omega$. Furthermore, using Dirichlet's prime number theorem we are able to count the number of isomorphism types of such Hopf algebras. More precisely, if $d = {\rm gcd}(m,\nu(n))$ and $\frac{\nu(n)}{d} = p_1^{\alpha_1} \cdots p_r^{\alpha_r}$ is the prime decomposition of~$\frac{\nu(n)}{d}$ then the number of types of Hopf algebras that factorize through $T_{m^{2}}(q)$ and $K[C_n]$ is equal to $(\alpha_1 + 1)(\alpha_2 + 1) \cdots (\alpha_r + 1)$, where $\nu (n)$ is the order of the group of $n$-th roots of unity in $K$. As a consequence of our approach, the automorphism groups of these Hopf algebras are described as well.}

\Keywords{bicrossed product; the factorization problem; classification of Hopf algebras}

\Classification{16T10; 16T05; 16S40}

\section{Introduction}
The factorization problem was considered in many settings ranging from groups \cite{acim, Ore, Takeuchi} to (co)algebras \cite{cap, CIMZ, Pena2}, Lie algebras \cite{majid3, Mic}, locally compact quantum groups \cite{VV} or fusion categories~\cite{SG}. We refer the reader to \cite{abm} and the references therein for a thorough overview on recent developments. The present paper is a contribution to the factorization problem for Hopf algebras. For two given Hopf algebras $A$ and $H$, it can be stated as follows:

\textit{Describe and classify up to an isomorphism all Hopf algebras that factorize through~$A$ and~$H$, i.e., all Hopf algebras $E$ for which there exist injective Hopf algebra maps $i\colon A \to E$ and \smash{$j\colon H\to E$} such that the following map is bijective:}
\begin{gather*}
A \otimes H \to E, \qquad a \otimes y \mapsto i(a) j(y).
\end{gather*}
Like many other notions or major structural results in Hopf algebras, the factorization problem originates in group theory~\cite{Ore}. Despite its elementary statement, studying the factorization problem presents many challenges even in the case of groups. Perhaps the most natural question to consider is that of describing and classifying groups which factorize through two finite cyclic groups. Surprisingly, after being investigated in
many papers it remains still an open question (see~\cite{acim} for a more detailed account on the subject). However, the important achievements obtained in the group setting inspired a new approach to this problem in the Hopf context as well. Probably the most influential papers in this direction were~\cite{zappa} and~\cite{Takeuchi} which introduce the bicrossed product associated to a matched pair of groups as an equivalent way of tackling the factorization problem. The generalization of this construction at the level of Hopf algebras proposed by Majid in~\cite{majid}
opened the way for a more systematic approach. More precisely, with this new construction at hand, the factorization problem comes down to computing all matched pairs of Hopf algebras between~$A$ and~$H$ and classifying the corresponding bicrossed products (see Section~\ref{prel} for the precise definitions). The line of inquiry proposed in \cite{abm} follows this path and has proved itself to be an efficient strategy for addressing the
classification of bicrossed products. The first argument supporting this idea is the example presented in~\cite{abm} which completely describes and classifies all bicrossed products between Sweedler's Hopf algebra and the group Hopf algebra $K[C_{n}]$, where $C_n$ denotes the cyclic group of order~$n$. Furthermore, the method introduced in \cite{abm} was successfully used in \cite{Gabi} and \cite{aa} in order to describe and classify all
Hopf algebras which factorize through two Sweedler's Hopf algebras and respectively two Taft Hopf algebras. Besides from being of interest in its own right by contributing to a~better understan\-ding of Hopf algebras and their classification, the study of matched pairs and their corresponding bicrossed products offers many interesting applications. For instance, the classification of bicrossed products obtained in~\cite{abm} is the main ingredient in describing the automorphism group of the Drinfel'd double of a~purely non-abelian finite group~\cite{marc1}. Furthermore, in a more recent paper, matched pairs of cocommutative Hopf algebras are used to construct solutions for the quantum Yang--Baxter equation (see \cite[Proposition~3.2]{agv}).

The present paper is in some sense a sequel to \cite{abm} and investigates Hopf algebras which factorize through the Taft algebra $T_{m^{2}}(q)$ and the group Hopf algebra $K[C_{n}]$. It is organized as follows: Section~\ref{prel} contains some background material needed afterwards. We review some useful notions such as that of matched pair of Hopf algebras and the corresponding bicrossed product as well as the statements of some important theorems. Section~\ref{2} contains the main results. Let $n, m \in \mathbb{N}$, $n \geqslant 2$, $m \geqslant 3$ and let $K$ be a field
which contains a primitive $m$-th root of unity~$q$. $U_n (K) = \{\omega \in K \,|\, \omega^n = 1 \}$ denotes the cyclic group of $n$-th roots of unity in $K$ of order $\nu (n) = |U_n (K)|$ and $T_{m^{2}}(q)$ is the Taft Hopf algebra. First we compute in Proposition~\ref{mapex1} all possible matched pairs $(T_{m^{2}}(q), K[C_n], \triangleleft, \triangleright)$. It turns out that these matched pairs are in bijection with the cyclic group $U_n (K)$: the right action $\triangleleft \colon K[C_{n}] \otimes T_{m^{2}}(q) \to K[C_{n}]$ is trivial while the left action $\triangleright \colon K[C_{n}] \otimes T_{m^{2}}(q) \to T_{m^{2}}(q)$ is implemented by an $n$-th root of unity~$\omega$. Corollary~\ref{primaclasi} shows that a Hopf algebra $E$ factorizes through $T_{m^{2}}(q)$ and $K[C_n]$ if and only if $E \cong T_{nm^{2}}^ {\omega}(q)$, where $T_{nm^{2}}^ {\omega}(q)$ is a~quantum group associated to an $n$-th root of unity $\omega$ which we completely describe by generators and relations. We point out that the Hopf algebras $T_{nm^{2}}^ {\omega}(q)$ are in particular pointed, of rank one. Pointed Hopf algebras were also investigated in \cite{K,S,W} where several classification results are obtained. However, they are not as explicit and detailed as those provided in the present paper and, moreover, the problem of counting the isomorphism types of these Hopf algebras is not considered. Theorem~\ref{izo} gives necessary and sufficient conditions for two Hopf algebras $T_{nm^{2}}^ {\omega}(q)$ and respectively $T_{nm^{2}}^ {\omega'}(q)$, $\omega, \omega ' \in U_n (K)$, to be isomorphic. The
preceding result, together with Dirichlet's prime number theorem, is then used to count the number of isomorphism types of the aforementioned
Hopf algebras. If $d = {\rm gcd}(m, \nu (n))$ and $\frac{\nu(n)}{d} = p_1^{\alpha_1} \cdots p_r^{\alpha_r}$ is the prime decomposition of
$ \frac{\nu(n)}{d}$ then the number of types of Hopf algebras that factorize through $T_{m^{2}}(q)$ and $K[C_n]$ is equal to $(\alpha_1 + 1)(\alpha_2 + 1) \cdots (\alpha_r + 1)$. This generalizes \cite[Theorem~5.10]{abm}. As a direct consequence of this approach we obtain in Theorem~\ref{graut} a~description of the automorphism group of~$T_{nm^{2}}^ {\omega}$.

\section{Preliminaries}\label{prel}
Throughout $K$ will be a field. For a positive integer $n$ we denote by $U_n (K) = \{ \omega \in K \,|\, \omega^n = 1\}$ the cyclic group of $n$-th roots of unity in $K$, and by $\nu (n) = |U_n (K)|$ its order. Obviously $\nu (n)$ is a divisor of $n$; if $\nu (n) = n$, then any generator of $U_n (K)$ is called a~primitive $n$-th root of unity. In the sequel~$C_{n}$ will be the cyclic group of order $n$ generated by $g$ and $K[C_{n}]$ the corresponding group Hopf algebra. For any~$i,j \in \mathbb{N}$, $\delta_{i, j}$ denotes the Kronecker delta.

Let $m \in \mathbb{N}$, $m \geq 2$. Whenever we deal with Taft algebras of order $m^{2}$ we assume that the base field $K$ contains a~primitive $m$-th root of unity $q$. $T_{m^{2}}(q)$ denotes the Taft Hopf algebra of order $m$ over $K$, which is generated as an algebra by two elements~$h$ and~$x$ subject to the relations $h^{m} = 1$, $x^{m} = 0$ and $xh = qhx$. The coalgebra structure and antipode are given as follows
\begin{gather*}
\Delta(h) = h \otimes h, \qquad \Delta(x) = x \otimes h + 1 \otimes x, \qquad \epsilon(h) = 1,\\
 \epsilon(x) = 0,\qquad S(h) = h^{-1},\qquad S(x) = -xh^{-1}.
\end{gather*}
that is, $h$ is a group like element while $x$ is $(h,1)$-primitive. Sweedler's Hopf algebra is obtained by considering $m=2$ and $q = -1$. It can be easily checked that $\{h^{i}x^{j}\}_{0 \leq i, j \leq m-1}$ is a $K$-linear basis of the Taft algebra, the set of group-like elements is
$\mathcal{G}\big(T_{m^{2}}(q)\big) = \{1, h, h^{2},\dots, h^{m-1}\}$ and the primitive elements $\mathcal{P}_{h^{j}, 1}\big(T_{m^{2}}(q)\big)$ are given as follows for any $j = \{0, 1,\dots, m-1\}$
\begin{gather*}\mathcal{P}_{h^{j}, 1} \big(T_{m^{2}}(q)\big) = \begin{cases}
\alpha\big(h^{j} - 1\big), & \text{if } j \neq 1,\\
\beta(h-1)+ \gamma x, & \text {if } j=1,
\end{cases}\qquad \text{for some} \ \alpha, \beta, \gamma \in K.\end{gather*}
Unless specified otherwise, all algebras, coalgebras, bialgebras, Hopf algebras, tensor products and homomorphisms are over~$K$. For a coalgebra $(C, \Delta, \varepsilon)$, we use Sweedler's $\Sigma$-notation: $\Delta(c) = c_{(1)}\otimes c_{(2)}$, $(I\otimes \Delta)\Delta(c) = c_{(1)}\otimes
c_{(2)}\otimes c_{(3)}$, etc.\ (summation understood). Let $A$ and $H$ be two Hopf algebras. $H$ is called a right $A$-module coalgebra if $H$ is a coalgebra in the monoidal category ${\mathcal M}_A $ of right $A$-modules, i.e., there exists $\triangleleft \colon H \otimes A \rightarrow H$ a morphism of coalgebras such that $(H,\triangleleft) $ is a right $A$-module. A morphism between two right $A$-module coalgebras $(H, \triangleleft)$ and $(H',\triangleleft')$ is a morphism of coalgebras $\psi\colon H \to H'$ that is also right $A$-linear. Furthermore, $\psi$~is called unitary if
$\psi (1_H) = 1_{H'}$. Similarly, $A$ is a left $H$-module coalgebra if $A$ is a coalgebra in the monoidal category of left $H$-modules, that is there exists $\triangleright \colon H \otimes A \to A$ a morphism of coalgebras such that $(A, \triangleright)$ is also a~left $H$-module. The actions $\triangleleft \colon H \otimes A \rightarrow H$, $\triangleright: H \otimes A \rightarrow A$ are called trivial if $y \triangleleft a = \varepsilon_A (a) y$ and respectively $y \triangleright a = \varepsilon_H(y) a$, for all $a\in A$ and $y\in H$.

A Hopf algebra $E$ \emph{factorizes} through two Hopf algebras $A$ and $H$ if there exist injective Hopf algebra maps $i \colon A \to E $
and $j \colon H\to E$ such that the following map is bijective:
\begin{gather*}
A \otimes H \to E,\qquad a \otimes y \mapsto i(a) j(y).
\end{gather*}
A crucial result which characterizes Hopf algebras that factorize through two given Hopf algebras in terms of bicrossed products was proved by Majid (see for instance \cite[Theorem~7.2.3]{majid2}). We recall briefly the construction of the \emph{bicrossed product} of two Hopf algebras. It was introduced by Majid in \cite[Proposition 3.12]{majid} under the name of \emph{double cross product}. Throughout we shall adopt the name of bicrossed product from \cite[Theorem~IX 2.3]{Kassel} since it has also been used for similar constructions in many other fields (see, e.g., the group case \cite{Takeuchi}). The main ingredient for constructing bicrossed products are the so-called matched pairs. A \textit{matched pair} of Hopf algebras is a~quadruple $(A, H, \triangleleft, \triangleright)$, where $A$ and $H$ are Hopf algebras, $\triangleleft \colon H \otimes A \rightarrow H$,
$\triangleright\colon H \otimes A \rightarrow A$ are coalgebra maps such that $(A, \triangleright)$ is a left $H$-module coalgebra, $(H, \triangleleft)$ is a~right $A$-module coalgebra and the following compatibilities hold for any $a, b\in A$, $y, z\in H$
\begin{gather}
y \triangleright 1_{A} {=} \varepsilon_{H}(h)1_{A}, \qquad 1_{H} \triangleleft a = \varepsilon_{A}(a)1_{H}, \label{mp1} \\
y \triangleright(ab) {=} (y_{(1)} \triangleright a_{(1)}) \big( \big(y_{(2)}\triangleleft a_{(2)}\big)\triangleright b \big),\label{mp2} \\
(y z) \triangleleft a {=} \big( y \triangleleft \big(z_{(1)}\triangleright a_{(1)}\big) \big) (z_{(2)} \triangleleft a_{(2)}),\label{mp3} \\
y_{(1)} \triangleleft a_{(1)} \otimes y_{(2)} \triangleright a_{(2)} {=} y_{(2)} \triangleleft a_{(2)} \otimes y_{(1)}\triangleright a_{(1)}.\label{mp4}
\end{gather}
If $(A, H, \triangleleft, \triangleright)$ is a matched pair of Hopf algebras then the corresponding \textit{bicrossed product of Hopf algebras} $A \bowtie H$ is the tensor coalgebra $A\otimes H$ with the multiplication and antipode given as follows
\begin{gather*}
(a \bowtie y) \cdot (b \bowtie z) = a \big(y_{(1)}\triangleright b_{(1)}\big) \bowtie \big(y_{(2)} \triangleleft b_{(2)}\big) z, \\
S_{A \bowtie H} ( a \bowtie y ) = S_H \big(y_{(2)}\big) \triangleright S_A \big(a_{(2)}\big) \bowtie S_H \big(y_{(1)}\big) \triangleleft S_A \big(a_{(1)}\big)
\end{gather*}
for all $a,b \in A$, $y,z\in H$, where we denote $a\otimes y$ by $a\bowtie y$.

Although there are many interesting examples of bicrossed products such as the Drinfel'd double \cite[Theorem~IX.3.5]{Kassel} or the generalized quantum double \cite[Example~7.2.6]{majid} we will restrict to presenting only one example which is important for our purposes, namely the semi-direct (smash) product of Hopf algebras.

\begin{Example} \label{exempleban}
Let $(A, \triangleright)$ be a left $H$-module coalgebra and consider $H$ as a right $A$-module coalgebra via the trivial action, i.e., $y \triangleleft a = \varepsilon_A(a) y$. Then $(A,H, \triangleleft, \triangleright)$ is a matched pair of Hopf algebras if and only if $(A, \triangleright)$ is also a left $H$-module algebra and the following compatibility condition holds for all $y \in H$ and $a\in A$
\begin{gather*}
y_{(1)} \otimes y_{(2)} \triangleright a = y_{(2)} \otimes y_{(1)} \triangleright a.
\end{gather*}
In this case, the associated bicrossed product $A\bowtie H = A\# H$ is the left version of the \emph{semi-direct $($smash$)$ product} of Hopf algebras as defined by Molnar~\cite{Mo} in the cocommutative Hopf algebra setting where the compatibility condition~\eqref{mp4} holds automatically. Thus, $A\# H$ is the tensor coalgebra $A\otimes H$, with the following multiplication
\begin{gather*}
(a \# y) \cdot (b \# z):= a \big(y_{(1)} \triangleright b \big) \# y_{(2)} z
\end{gather*}
for all $a, b\in A$, $y, z\in H$, where we denote $a\otimes y$ by $a\# y$. Similarly one can define the right version of the smash product of Hopf algebras by considering $A$ as a left $H$-module coalgebra via the trivial action, i.e., $y \triangleright a := \varepsilon_H(y) a$.
\end{Example}

The next result, proved by Majid, establishes the connection between Hopf algebras which factorize through two given Hopf algebras and bicrossed products.

\begin{Theorem}[{\cite[Theorem 7.2.3]{majid2}}]\label{carMaj}
Consider two Hopf algebras $A$ and $H$. A Hopf algebra $E$ factorizes through $A$ and $H$ if and only if there exists a matched pair of Hopf algebras $(A, H, \triangleleft, \triangleright)$ and an isomorphism of Hopf algebras $E \cong A \bowtie H$.
\end{Theorem}
In light of Theorem~\ref{carMaj} the factorization problem for Hopf
algebras can be restated in terms of matched pairs and bicrossed products as follows:

\textit{For two given Hopf algebras $A$ and $H$, describe the set of all matched pairs $(A, H, \triangleright, \triangleleft)$ and classify
up to an isomorphism the corresponding bicrossed products $A \bowtie H$.}

The strategy we employ for classifying semi-direct products will rely on the following result which is a special case of \cite[Corollary~3.3]{abm}.

\begin{Theorem}\label{toatemorf}
Consider $A \# H$ and $A \#' H$ to be two semi-direct products associated to the left actions $\triangleright \colon H\otimes A \to A$ and
respectively $\triangleright' \colon H\otimes A \to A$. Then there exists a~bijective correspondence between the set of all morphisms of Hopf
algebras $\psi \colon A \# H \to A \# ' H $ and the set of all quadruples $(u, p, r, v)$, consisting of two unitary coalgebra maps $u\colon A \to A$ and $r\colon H \rightarrow A$, and two Hopf algebra maps $p\colon A \to H$ and $v\colon H \rightarrow H$ subject to the following compatibilities
\begin{gather}
u\big(a_{(1)}\big) \otimes p\big(a_{(2)}\big) = u\big(a_{(2)}\big) \otimes p\big(a_{(1)}\big),\label{C1}\\
r\big(t_{(1)}\big) \otimes v\big(t_{(2)}\big) = r\big(t_{(2)}\big) \otimes v\big(t_{(1)}\big),\label{C2}\\
u(ab) = u\big(a_{(1)}\big) \big( p \big(a_{(2)}\big) \triangleright' u(b) \big),\label{C3}\\
r(tw) = r\big(t_{(1)}\big) \big(v\big(t_{(2)}\big) \triangleright' r(w)\big),\label{C4}\\
r\big(t_{(1)}\big) \big(v\big(t_{(2)}\big) \triangleright' u(b) \big) = u \big(t_{(1)} \triangleright b_{(1)}\big) \big( p \big(t_{(2)}
\triangleright b_{(2)}\big) \triangleright' r\big(t_{(3)}\big) \big) ,\label{C5}\\
v(t) p (b)= p \big(t_{(1)} \triangleright b\big) v \big(t_{(2)}\big)\label{C6}
\end{gather}
for all $a, b \in A$, $t, w \in H$.

Under the above correspondence the morphism of Hopf algebras $\psi\colon A \# H \to A \# ' H $ corresponding to $(u, p, r, v)$ is given by
\begin{gather*}%\label{morfbicros}
\psi_{(u, p,r, v)}(a \# t) = u\big(a_{(1)}\big) \big( p\big(a_{(2)}\big) \triangleright' r\big(t_{(1)}\big) \big) \#' p\big(a_{(3)}\big) v\big(t_{(2)}\big)
\end{gather*}
for all $a \in A$ and $t\in H$.
\end{Theorem}

Finally, for the convenience of the reader we include here a straightforward result which will be intensively used in computing all matched pairs between $T_{m^{2}}(q)$ and $K[C_{n}]$.
\begin{Lemma}[{\cite[Lemma 4.1]{abm}}]\label{primitive}
Let $(A, H, \triangleleft, \triangleright)$ be a matched pair of Hopf algebras, $a, b \in \mathcal{G}(A)$ and $t, w\in \mathcal{G}(H)$. Then
\begin{enumerate}\itemsep=0pt
\item[$(1)$] $t \triangleright a \in \mathcal{G}(A)$ and $t\triangleleft a \in \mathcal{G}(H)$;

\item[$(2)$] if $c \in \mathcal{P}_{a, b}(A)$, then $t \triangleleft c \in \mathcal{P}_{t\triangleleft a, t\triangleleft b} (H)$ and
$t \triangleright c \in \mathcal{P}_{t\triangleright a, t\triangleright b} (A)$;

\item[$(3)$] if $z \in \mathcal{P}_{t, w}(H)$, then $z \triangleleft a \in \mathcal{P}_{t\triangleleft a, w\triangleleft a} (H)$ and $z \triangleright a \in \mathcal{P}_{t\triangleright a, w\triangleright a} (A)$.
\end{enumerate}
In particular, if $c$ is an $(a, 1)$-primitive element of $A$, then $t \triangleright c$ is an $(t\triangleright a, 1)$-primitive element of $A$ and $t \triangleleft c$ is an $(t\triangleleft a, 1)$-primitive element of~$H$.
\end{Lemma}

\section{Main results}\label{2}

In this section we deal with the description and classification of Hopf algebras which factorize through $T_{m^{2}}(q)$ and $K[C_{n}]$ for any $m, n \in \mathbb{N}$, $n \geqslant 2$, $m \geqslant 3$. The case $m = 2$ was recently considered in~\cite{abm}. Throughout, the field $K$ will be assumed to contain a primitive $m$-th root of unity $q$. We start by describing all possible matched pairs $(T_{m^{2}}(q), K[C_n], \triangleleft,
\triangleright)$: exactly as in~\cite{abm}, it turns out that the aforementioned matched pairs are in bijection with the cyclic group~$U_n (K)$ of
$n$-th roots of unity in~$K$. More precisely, the right action $\triangleleft \colon K[C_{n}] \otimes T_{m^{2}}(q) \to K[C_{n}]$ is trivial while
the left action $\triangleright \colon K[C_{n}] \otimes T_{m^{2}}(q) \to T_{m^{2}}(q)$ is implemented by an $n$-th root of unity as described below.

\begin{Proposition} \label{mapex1}
There exists a bijective correspondence between the set of all matched pairs $(T_{m^{2}}(q), K[C_n], \triangleleft, \triangleright)$ and $U_n (K)$. More precisely, the matched pair $(\triangleleft, \triangleright)$ corresponding to an $n$-th root of unity $\omega \in U_n (K)$ is given by
\begin{gather}\label{mp}
g^{i} \triangleright h^{j} x^{k} = \omega^{ik} h^{j} x^{k}, \qquad g^{i} \triangleleft h^{j} x^{k} = g^{i} \varepsilon\big(x^{k}\big)
\end{gather}
for all $i = 0,\dots, n-1$ and $j, k = 0,\dots, m-1$.
\end{Proposition}

\begin{proof} Let $(T_{m^{2}}(q), K[C_n], \triangleleft, \triangleright)$ be a~matched pair. We start by proving that $g \triangleright h = h$.
Indeed, by Lemma~\ref{primitive} we have $g \triangleright h \in \mathcal{G}\big(T_{m^{2}}(q)\big) = \{1, h,\dots,h^{m-1}\}$. Suppose first that $g \triangleright h = 1$. Then, by induction we have $1 = g^{n} \triangleright h = h$ which is obviously a contradiction. Assume now that $g \triangleright h = h^{t}$, where $t \in \{2, 3,\dots, m-1\}$. Now Lemma~\ref{primitive} implies $g \triangleright x \in \mathcal{P}_{h^{t}, 1}\big(T_{m^{2}}(q)\big)$. Then, as $t \neq 1$ we obtain $g \triangleright x = \alpha(1-h^{t})$ for some $\alpha \in K$. Hence, $g^{2} \triangleright x = g \triangleright (g \triangleright x) = \alpha - \alpha g \triangleright h^{t}$. Again by induction we arrive at $x = g^{n} \triangleright x = \alpha - \alpha g^{n-1} \triangleright h^{t}$ which leads to another contradiction since Lemma~\ref{primitive} yields $g^{n-1} \triangleright h^{t} \in \mathcal{G}\big(T_{m^{2}}(q)\big) = \{1, h,\dots,h^{m-1}\}$. Hence $g \triangleright h = h$ as desired. Furthermore, this
also implies that $g \triangleright x = \alpha(1-h)+ \beta x$ for some~$\alpha, \beta \in K$. Using induction again we obtain
\begin{gather*}
g^{n} \triangleright x = \alpha\big(1 + \beta + \cdots + \beta^{n-1}\big) (1-h) + \beta^{n} x.
\end{gather*}
As we also have $g^{n} \triangleright x = x$ it follows that
\begin{gather*}
\beta^{n} = 1, \qquad \alpha\big(1 + \beta + \cdots + \beta^{n-1}\big)=0.
\end{gather*}
Next we look at the right action $\triangleleft$. Again by Lemma~\ref{primitive} we have $g \triangleleft h \in \mathcal{G}\big(K[C_n]\big) = \{1, g,\dots, g^{n-1}\}$, therefore $g \triangleleft h = g^{t}$ for some $t \in \{0,1, \dots, n-1\}$. If $g \triangleleft h = 1$ we obtain by induction
$1 = g \triangleleft h^{n} = g \triangleleft 1 = g$ which is a~contradiction. Thus $g \triangleleft h = g^{t}$ for some $t \in \{1, 2,\dots, n-1\}$. Moreover, we also have $g \triangleleft x \in \mathcal{P}_{g^{t}, g} \big(K[C_{n}]\big)$, i.e., $g \triangleleft x = \mu(g - g^{t})$ for some $\mu \in K$. Applying the compatibility condition~\eqref{mp4} for the pair $(g, x)$ yields
\begin{gather*}
\mu g \otimes h - \mu g^{t} \otimes h + \alpha g \otimes 1 - \alpha g \otimes h + \beta g \otimes x \\
\qquad{} = \alpha g^{t} \otimes 1 - \alpha g^{t} \otimes h + \beta g^{t} \otimes x + \mu g \otimes 1 - \mu g^{t} \otimes 1.
\end{gather*}
Now if $t \neq 1$ then we must have $\beta = 0$ which again contradicts $\beta^{n} = 1$. Hence $t = 1$ which implies $g \triangleleft h = g$ and respectively $g \triangleleft x = 0$. Moreover, the compatibility condition~\eqref{mp4} is now trivially fulfilled.

Next we make use of the compatibility condition~\eqref{mp2}. More precisely we have
\begin{gather*}
g \triangleright hx \stackrel{\eqref{mp2}} {=} (g \triangleright h) \big((g \triangleleft h) \triangleright x\big)
= h (g \triangleright x) = \alpha h - \alpha h^{2} + \beta hx,\\
g \triangleright xh \stackrel{\eqref{mp2}} {=} \big(g \triangleright x_{(1)}\big) \big(\big(g \triangleleft x_{(2)}\big) \triangleright h\big)
= (g \triangleright x) \big((g \triangleleft h) \triangleright h\big) = \alpha h - \alpha h^{2} + \beta xh.
\end{gather*}
As $xh = q hx$, putting all the above together gives
\begin{gather*}
q\alpha h - q\alpha h^{2} + q\beta hx = \alpha h - \alpha h^{2} + \beta xh,
\end{gather*}
and since $q \neq 1$ we obtain $\alpha = 0$. Thus $g \triangleright x = \beta x$. To summarize, for all $i \in \{0, 1,\dots, n-1\}$ and $j \in \{0, 1,\dots, m-1\}$ we have
\begin{gather*}
g \triangleleft h^{j} = g, \qquad g \triangleleft x^{j} = g \varepsilon\big(x^{j}\big), \qquad g^{i} \triangleright h = h, \qquad g^{i} \triangleright x = \beta^{i} x,\qquad {\rm where} \quad \beta \in U_{n}(K).
\end{gather*}
Now using induction and the compatibility conditions~\eqref{mp2} and respectively~\eqref{mp3} we obtain the formulae in~\eqref{mp}. To start with, using~\eqref{mp3} we obtain $g^{i} \triangleleft h = g^{i}$ for all $i \in \{0, 1,\dots, n-1\}$. Indeed, we have
\begin{gather*}
g^{2} \triangleleft h \stackrel{\eqref{mp3}} = \big(g \triangleleft (g \triangleright h)\big) (g \triangleleft h)= g^{2},\\
g^{3} \triangleleft h \stackrel{\eqref{mp3}} = \big(g^{2} \triangleleft (g \triangleright h)\big) (g \triangleleft h)= \big(g^{2} \triangleleft h\big)(g \triangleleft h) = g^{3},
\end{gather*}
and by induction we arrive at the desired conclusion. Furthermore, as $\triangleleft$ is a right action of $T_{m^{2}}(q)$ on $K[C_{n}]$, we obtain
\begin{gather}\label{001}
g^{i} \triangleleft h^{j} = g^{i},\qquad \text{for all} \quad i \in \{0, 1,\dots, n-1\} \quad {\rm and}\quad j \in \{0, 1,\dots, m-1\}.
\end{gather}
Now induction together with~\eqref{mp2} and~\eqref{001} yields
\begin{gather}\label{002}
g^{i} \triangleright h^{j} = h^{j}, \qquad \text{for all} \quad i \in \{0, 1,\dots, n-1\} \quad {\rm and} \quad j \in \{0, 1,\dots, m-1\}.
\end{gather}
Next we use~\eqref{mp3} in order to prove that $g^{i} \triangleleft x =0$. Indeed, we have
\begin{gather*}
g^{2} \triangleleft x \stackrel{\eqref{mp3}} {=} \big(g \triangleleft \big(g \triangleright x_{(1)}\big)\big) \big(g \triangleleft
x_{(2)}\big) = \big(g \triangleleft (g \triangleright x)\big) (g \triangleleft h) + \big(g \triangleleft (g \triangleright 1)\big) (g \triangleleft x)=0, \\
g^{3} \triangleleft x \stackrel{\eqref{mp3}} {=} \big(g^{2} \triangleleft \big(g \triangleright x_{(1)}\big)\big) \big(g \triangleleft
x_{(2)}\big)= \big(g^{2} \triangleleft (g \triangleright x)\big) (g \triangleleft h) + \big(g^{2} \triangleleft (g \triangleright 1)\big) (g \triangleleft x)=0,
\end{gather*}
and the conclusion follows again by induction. Moreover, we actually have
\begin{gather*}%\label{003}
g^{i} \triangleleft x^{j} = g^{i} \varepsilon\big(x^{j}\big), \qquad \text{for all} \quad i \in \{0, 1,\dots, n-1\} \quad {\rm and} \quad j \in \{0, 1,\dots, m-1\}.
\end{gather*}
Finally, we will use~\eqref{mp3} in order to prove that
\begin{gather}\label{003}
g^{i} \triangleright x^{j} = \beta^{ij}x^{j}, \qquad \text{for all} \quad i \in \{0, 1,\dots, n-1\} \quad {\rm and} \quad j \in \{0, 1,\dots, m-1\}.
\end{gather}
To this end, we have
\begin{gather*}
g^{i} \triangleright x^{2} \stackrel{\eqref{mp2}} {=} \big(g^{i} \triangleright x_{(1)}\big)\big( \big(g^{i} \triangleleft
x_{(2)}\big)\triangleright x\big) \\
\hphantom{g^{i} \triangleright x^{2}}{} \ = \big(g^{i} \triangleright x\big)\big( \big(g^{i} \triangleleft h\big)\triangleright x\big) + \big(g^{i} \triangleright 1\big)\big( \big(g^{i} \triangleleft x\big)\triangleright x\big) = \beta^{2i}x^{2},\\
g^{i} \triangleright x^{3} \stackrel{\eqref{mp2}} {=} \big(g^{i} \triangleright x_{(1)}\big)\big( \big(g^{i} \triangleleft
x_{(2)}\big)\triangleright x^{2}\big)\\
 \hphantom{g^{i} \triangleright x^{3}}{} \ = \big(g^{i} \triangleright x\big)\big( \big(g^{i} \triangleleft h\big)\triangleright x^{2}\big) + \big(g^{i} \triangleright 1\big)\big( \big(g^{i} \triangleleft x\big)\triangleright x^{2}\big) =\beta^{3i}x^{3},
\end{gather*}
and as before the conclusion follows by induction. Putting all the above together we obtain
\begin{gather*}
g^{i} \triangleright h^{j} x^{k} \stackrel{\eqref{mp2}} {=} \big(g^{i} \triangleright h^{j}\big)\big( \big(g^{i} \triangleleft h^{j}\big)\triangleright x^{k}\big) \stackrel{\eqref{001}, \, \eqref{002}} {=} h^{j} \big(g^{i} \triangleright x^{k}\big)\stackrel{\eqref{003}} {=} \beta^{ik} h^{j} x^{k}.
\end{gather*}
Similarly one can prove the second part of \eqref{mp} and the proof is now finished.
\end{proof}

Next we show that a Hopf algebra $E$ factorizes through $T_{m^{2}}(q)$ and $K[C_n]$ if and only if $E \cong T_{nm^{2}}^{\omega}(q)$, where $T_{nm^{2}}^ {\omega}(q)$ is a quantum group associated to an $n$-th root of unity $\omega$ as depicted below.

\begin{Corollary}\label{primaclasi}
A Hopf algebra $E$ factorizes through $T_{m^{2}}(q)$ and $K[C_n]$ if and only if $E \cong T_{nm^{2}}^ {\omega}(q)$, for some $\omega \in U_n (K)$, where we denote by $T_{nm^{2}}^{\omega}(q)$ the Hopf algebra generated by $g$, $h$ and $x$ subject to the relations
\begin{gather*}
g^{n} = h^{m} = 1,\qquad x^{m} = 0,\qquad x h = q h x, \qquad h g = g h,\qquad g x = \omega x g
\end{gather*}
with the coalgebra structure and antipode given by
\begin{gather*}
\Delta(g) = g \otimes g, \qquad \Delta(h) = h \otimes h, \qquad \Delta(x) = x \otimes h + 1 \otimes x,\\
\varepsilon(h) = \varepsilon(g) = 1, \qquad \varepsilon(x) = 0, \qquad S(h) = h^{m-1}, \qquad S(x) = -xh^{m-1}, \qquad S(g) = g^{n-1}.
\end{gather*}
\end{Corollary}

\begin{proof} By Theorem~\ref{carMaj} any Hopf algebra~$E$ which factorizes through $T_{m^{2}}(q)$ and $K[C_n]$ is isomorphic to a bicrossed product
between the aforementioned Hopf algebras. Furthermore, all bicrossed products between $T_{m^{2}}(q)$ and $K[C_n]$ are described in Proposition~\ref{mapex1}: the left action is completely determined by an $n$-th root of unity $\omega$ as in~\eqref{mp} while the right action is trivial. Therefore, $E$ is in fact isomorphic to a smash product $T_{m^{2}}(q) \# K[C_n]$ associated to the left action defined in~\eqref{mp}. Now, up to canonical identification, $T_{m^{2}}(q) \# K[C_n]$ is generated as an algebra by $h = h \# 1$, $x = x \# 1$ and $g = 1\# g$. Hence, we have
\begin{gather*}
gh = (1\# g) (h \# 1)= g \triangleright h \# g \triangleleft h = h \# g = (h \# 1)(1 \# g) = hg,\\
g x = (1 \# g) (x \# 1) = g \triangleright x_{(1)}
\# g \triangleleft x_{(2)} = g \triangleright x \# g \triangleleft h + g \triangleright 1 \# g \triangleleft x =\omega x \# g = \omega x g. \tag*{\qed}
\end{gather*}\renewcommand{\qed}{}
\end{proof}
\begin{Remark}Let $\omega \in U_{n}(K)$. It can be easily seen that a $K$-basis in $T_{nm^{2}}^ {\omega}$ is given by $\{ g^{i}, h^{j} g^{i},
 x^{k} g^{i}, h^{j} x^{k} g^{i} \,| \,i = 0, 1,\dots, n-1, \,{\rm and}\, j, k = 1, 2,\dots, m-1 \}$. Obviously $T_{nm^{2}}^ {\omega}$ is a Hopf algebra of dimension $nm^{2}$ over~$K$.
\end{Remark}

Our next result provides the classification of Hopf algebras which factorize through the Taft algebra and the group Hopf algebra $k[C_{n}]$. To start with, if the Hopf algebras $T_{nm^{2}}^{\omega}$ and $T_{kl^{2}}^ {\omega '}$ are isomorphic it is straightforward to see, by comparing the order of their group of group-like elements and dimensions, that $n=k$ and $m=l$. Therefore, we will focus on finding necessary and sufficient conditions for two Hopf algebras $T_{nm^{2}}^ {\omega}$ and $T_{nm^{2}}^ {\omega '}$ to be isomorphic.

\begin{Theorem}\label{izo} Let $\xi$ be a generator of $U_n (K)$ and $t, t' \in \{ 0, 1,\dots, \nu (n) -1 \}$. Then the Hopf algebras $T_{nm^{2}}^ {\xi^{t}}$ and $T_{nm^{2}}^ {\xi^{t'}}$ are isomorphic if and only if there exist $l \in \{ 0, 1,\dots, m-1 \}$, $s \in \{ 0, 1,\dots, n-1 \}$ such that $(s, n) = 1$ and $\xi^{st' - t} = q^{l}$.
\end{Theorem}

\begin{proof} We know from Corollary~\ref{primaclasi} that $T_{nm^{2}}^ {\xi^{t}}$ (respectively $T_{nm^{2}}^ {\xi^{t'}}$) is the smash product corresponding to the root of unity $\xi^{t}$ (respectively $\xi^{t'}$) as in Proposition~\ref{mapex1}. By Theorem~\ref{toatemorf} the set of all Hopf algebra morphisms between $T_{nm^{2}}^ {\xi^{t}}$ and $T_{nm^{2}}^ {\xi^{t'}}$ is parameterized by the quadruples $(u, p, r, v)$, where $u\colon T_{m^{2}}(q) \to T_{m^{2}}(q)$, $r\colon K[C_{n}] \to T_{m^{2}}(q)$ are unitary coalgebra maps and $p\colon T_{m^{2}}(q) \to K[C_{n}]$, $v\colon K[C_{n}] \to K[C_{n}]$ are Hopf algebra maps satisfying the compatibility conditions~\eqref{C1}--\eqref{C6}. Our aim is to describe completely the above quadruples for which the corresponding morphism of Hopf algebras $\psi_{(u,p,r,v)}$ is bijective. We start by describing the two Hopf algebra maps $p$ and $v$. Obviously, any Hopf algebra map $v\colon K[C_{n}] \to K[C_{n}]$ is completely determined by an integer $s \in \{ 0, 1, \dots, n-1 \}$ such that $v(g) = g^{s}$. Similarly, using Lemma~\ref{primitive} we get $p(h) = g^{c}$ for some $c \in \{ 0, 1, \dots, n-1 \}$ such that $n \,|\, mc$. Furthermore, $p(x) = \lambda(g^{c} - 1)$ for some scalar $\lambda \in K$. Now imposing the compatibility condition $p(xh) = q p(hx)$ to hold true we obtain $\lambda g^{c}(g^{c} - 1) (1-q) = 0$ and since $q \neq 1$ we are lead to $\lambda(g^{c} - 1) = 0$, that is $p(x) = 0$. Now we turn to the two coalgebra maps $u\colon T_{m^{2}}(q) \to T_{m^{2}}(q)$, $r\colon K[C_{n}] \to T_{m^{2}}(q)$. We obviously have $r(1) = 1$, $r(g) = h^{l}$ for some $l \in \{ 0, \dots, m-1 \}$. The compatibility condition~\eqref{C4} gives
\begin{gather*}
r\big(g^{2}\big) = r(g)\big(v(g) \triangleright ' r(g)\big) = h^{l} \big(g^{s} \triangleright ' h^{l}\big) = h^{2l}.
\end{gather*}
Using induction and~\eqref{C4} we get $r(g^{i}) = h^{il}$ for any $i \in \mathbb{N}$. In particular, we have $1 = r(g^{n}) = h^{nl}$ and therefore $m \,|\, nl$. Finally we are left to describe the coalgebra maps $u\colon T_{m^{2}}(q) \to T_{m^{2}}(q)$. We have $u(1) = 1$ and $u(h)=h^{d}$ for some $d \in \{ 0, 1,\dots, m-1 \}$. Now we make use of the compatibility condition~\eqref{C3}
\begin{gather*}
u\big(h^{2}\big) = u(h) \big(p(h) \triangleright ' u(h)\big) = h^{d} \big(g^{c} \triangleright ' h^{d}\big) = h^{2d}.
\end{gather*}
As before, using induction and~\eqref{C3} we get $u(h^{j}) = h^{jd}$ for any $j \in \mathbb{N}$. Lemma~\ref{primitive} gives $u(x) \in
\mathcal{P}_{h^{d}, 1} \big(T_{m^{2}}(q)\big) = \begin{cases}
\alpha(h^{d}-1), & \text{if } d \neq 1,\\
\alpha(h-1)+ \gamma x, & \text{if } d=1,
\end{cases}$ for some $\alpha, \gamma \in K$. Suppose first that $d \neq 1$ and thus $u(x) = \alpha(h^{d}-1)$. Then we have
\begin{gather*}
u(hx) \stackrel{\eqref{C3}} {=} u(h) \big(p(h) \triangleright ' u(x)\big) = h^{d} \big(g^{c} \triangleright ' \alpha\big(h^{d}-1\big)\big) = \alpha h^{d} \big(h^{d}-1\big),\\
u(xh) \stackrel{\eqref{C3}} {=} u\big(x_{(1)}\big) \big(p\big(x_{(2)}\big) \triangleright ' u(h)\big) = u(x) \big(p(h) \triangleright ' u(h)\big) + p(x)
\triangleright ' u(h) = \alpha h^{d} \big(h^{d}-1\big).
\end{gather*}
Now $u(xh) = q u(hx)$ gives $\alpha (q-1) h^{d}(h^{d}-1) = 0$ and since $q \neq 1$ we obtain $\alpha(h^{d}-1) = 0$. Therefore $u(x) = 0$. However, we will see that in this case the corresponding Hopf algebra morphism $\psi_{(u, p, r, v)}$ is not an isomorphism. Indeed we have
\begin{gather*}
\psi_{(u, p, r, v)}(x \# 1) = u\big(x_{(1)}\big) \big(p\big(x_{(2)}\big) \triangleright ' r(1)\big) \# p(x_{(3)}) v(1),\\
\hphantom{\psi_{(u, p, r, v)}(x \# 1)}{} = u\big(x_{(1)}\big) \# p\big(x_{(2)}\big)= u(x) \# p(h) + 1 \# p(x) = 0,
\end{gather*}
which is a contradiction. Therefore we must have $d = 1$ and $u(x) = \alpha (h-1) + \gamma x$. Hence, we obtain
\begin{gather*}
u(hx) \stackrel{\eqref{C3}} {=} u(h) \big(p(h) \triangleright '
u(x)\big) = h \big[g^{c} \triangleright ' \big(\alpha(h-1)+
\gamma x \big)\big] = \alpha h(h-1)+ \gamma \xi^{t'c} hx,\\
u(xh) \stackrel{\eqref{C3}} {=} u\big(x_{(1)}\big) \big(p\big(x_{(2)}\big)
\triangleright ' u(h)\big) = u(x) \big(p(h) \triangleright ' u(h)\big) + p(x) \triangleright ' u(h) \\
\hphantom{u(xh)}{} \ = \big[\alpha(h-1) + \gamma x\big](g^{c} \triangleright ' h) = \alpha h(h-1) + \gamma q hx.
\end{gather*}
Now $u(xh) = q u(hx)$ gives $\alpha = q \alpha$ and $q \gamma \xi^{t'c} = q \gamma$. As $q \neq 1$ we obtain $\alpha = 0$. Thus $\gamma \neq 0$ (otherwise we would have $u(x) = 0$ which leads to the same contradiction as in the previous case) and so $\xi^{t'c} = 1$. Hence, $u$ is given as follows
\begin{gather*}
u(h^{i}) = h^{i}, \qquad u\big(x^{i}\big) = \gamma^{i} x^{i}, \qquad \gamma \in K^{*}, \qquad i \in \{0, 1,\dots, m-1\}.
\end{gather*}
Putting all together, the quadruple $(u, p, r, v)$ is given as follows for all $i, j \in \mathbb{N}$
\begin{gather}
v\big(g^{i}\big) = g^{is}, \qquad {\rm with} \quad s \in \{ 0,\dots, n-1 \}, \nonumber\\ %\label{cuad1}\\
p\big(h^{i}x^{j}\big) = g^{ic} \delta_{0, j}, \qquad {\rm with} \quad c \in \{ 0,\dots, n-1 \} \quad \text{such that} \quad n \,|\, mc \quad {\rm and} \quad \xi^{t'c} = 1,\nonumber\\ %\label{cuad2}\\
r\big(g^{i}\big) = h^{il},\qquad {\rm with} \quad l \in \{ 0,\dots, m-1 \} \quad \text{such that} \quad m \,|\, nl,\label{cuad3}\\
u\big(h^{i}\big) = h^{i}, \qquad u\big(x^{i}\big) = \gamma^{i} x^{i}, \qquad {\rm with} \quad \gamma \in K^{*}. \nonumber %\label{cuad4}
\end{gather}
We are still left to check the compatibilities~\eqref{C1},~\eqref{C5} and~\eqref{C6}. Remark that~\eqref{C2} is trivially fulfilled as $K[C_{n}]$ is cocommutative.

\eqref{C1} is trivially fulfilled for $a = h$ while for $a = x$ comes down to the following
\begin{gather*}
 u\big(x_{(1)}\big) \otimes p\big(x_{(2)}\big) = u\big(x_{(2)}\big) \otimes p\big(x_{(1)}\big) \\
\qquad{} \Leftrightarrow \quad u(x) \otimes p(h) + u(1) \otimes p(x) = u(h) \otimes p(x) + u(x) \otimes p(1) \\
\qquad{} \Leftrightarrow \quad \gamma x \otimes g^{c} = \gamma x \otimes 1,
\end{gather*}
and since $\gamma \neq 0$ we obtain $c = 0$ and thus $p$ is the trivial map, i.e., $p(a) = \varepsilon(a) 1$ for all $a \in T_{m^{2}}(q)$. Now~\eqref{C6} is trivially fulfilled while~\eqref{C5} comes down to the following
\begin{gather*}
r\big(t_{(1)}\big) \big(v\big(t_{(2)}\big) \triangleright' u(b) \big) {=} u \big(t_{(1)} \triangleright b\big) r\big(t_{(2)}\big).
\end{gather*}
By setting $t = g$ and $b = x$ we obtain
\begin{gather*}
 r(g)\big(v(g) \triangleright' u(x) \big) = u(g \triangleright x) r(g)\quad
\!\Leftrightarrow\! \quad h^{l} (g^{s} \triangleright' \gamma x) = u\big(\xi^{t} x\big)h^{l}\quad
\!\Leftrightarrow\!\quad \gamma \xi^{st'} h^{l} x = \gamma \xi^{t} q^{l} h^{l} x,\!
\end{gather*}
and thus $\xi^{st' - t} = q^{l}$. This last equality together with $\nu(n) \,|\, n$ gives $m \,|\, nl$. Thus, the compatibility condition in~\eqref{cuad3} is trivially fulfilled.

To summarize, we have proven that if the Hopf algebras $T_{nm^{2}}^ {\xi^{t}}$ and $T_{nm^{2}}^ {\xi^{t'}}$ are isomorphic then there exist $l \in \{ 0, 1,\dots, m-1 \}$ and $s \in \{ 0, 1,\dots, n-1 \}$ such that $\xi^{st' - t} = q^{l}$. In order to complete the first part of the proof we are left to show that $(s, n) = 1$. To this end, $v$ is bijective as a consequence of $\psi$ being bijective. Indeed, if $\psi^{-1}$ is the inverse of $\psi$ then using the same arguments as above one can easily prove that there exists a unitary coalgebra map $\overline{r}\colon K[C_{n}] \to T_{m^{2}}(q)$ and a Hopf algebra map $\overline{v}\colon K[C_{n}] \to K[C_{n}]$ such that for any $z \in K[C_{n}]$ we have $\psi^{-1} (1 \# z) = \overline{r}(z_{(1)}) \# \overline{v}(z_{(2)})$. This gives
\begin{gather*}
1 \# z = \psi \circ \psi^{-1} (1 \# z) = u \big(\overline{r}\big(z_{(1)}\big)\big) r \big(\overline{v}\big(z_{(2)}\big)\big) \# v\big(\overline{v}\big(z_{(3)}\big)\big).
\end{gather*}
By applying $\varepsilon \otimes {\rm Id}_{K[C_{n}]}$ to the above identity yields $v \circ \overline{v}= {\rm Id}_{K[C_{n}]}$. Similarly one can prove that $\overline{v} \circ v= {\rm Id}_{K[C_{n}]}$ and we can conclude that $v$ is indeed an isomorphism. Obviously $v$ is bijective if and only if $(s, n) = 1$.

Assume now that there exist $l \in \{ 0, 1,\dots, m-1 \}$, $s \in \{ 0, 1,\dots, n-1 \}$ such that $\xi^{st' - t} = q^{l}$ and $(s, n) = 1$. Consider two unitary coalgebra maps $u_{\gamma}\colon T_{m^{2}}(q) \to T_{m^{2}}(q)$, $r_{l}\colon K[C_{n}] \to T_{m^{2}}(q)$ and a~Hopf algebra map $v_{s}\colon K[C_{n}] \to K[C_{n}]$ defined as follows for all $i, j \in \{0, 1,\dots, m-1\}$ and $k \in \{0, 1,\dots, n-1\}$
\begin{gather*}
u_{\gamma}\big(h^{i}x^{j}\big) = \gamma^{j} h^{i}x^{j}, \qquad {\rm where}\quad \gamma \in K^{*}, \qquad r_{l}\big(g^{k}\big) = h^{kl}, \qquad v_{s}\big(g^{k}\big) = g^{ks}.
\end{gather*}
We claim that the following map is an isomorphism of Hopf algebras
\begin{gather*}
\psi_{(u_{\gamma}, \varepsilon, r_{l}, v_{s})}\colon \ T_{nm^{2}}^ {\xi^{t}} \to T_{nm^{2}}^ {\xi^{t'}}, \qquad \psi_{(u_{\gamma}, \varepsilon, r_{l}, v_{s})}(a \# y) = u_{\gamma}(a) r_{l}\big(y_{(1)}\big) \# v_{s}\big(y_{(2)}\big).
\end{gather*}
To start with, $\psi_{(u_{\gamma}, \varepsilon, r_{l}, v_{s})}$ is a Hopf algebra map according to the first part of the proof. We only need to show that $\psi_{(u_{\gamma}, \varepsilon, r_{l}, v_{s})}$ is bijective as well. To this end, since $(s, n) = 1$, there exist $\tau,\mu \in \mathbb{Z}$ such that $s \tau + n \mu = 1$. Furthermore, there exist unique integers $\alpha$, $\beta$, $\tau_{1}$, $\tau_{2}$ such that
\begin{gather}\label{impr}
\tau = \alpha n + \tau_{1} \qquad {\rm and} \qquad l \tau = \beta m + \tau_{2},
\end{gather}
where $\tau_{1} \in \{0, 1,\dots, n-1\}$, $\tau_{2} \in \{0, 1,\dots, m-1\}$. It will be worth our while to point out the following easy consequence of~\eqref{impr}
\begin{gather}\label{impr2}
l\tau_{1} = \beta m - \alpha nl + \tau_{2}.
\end{gather}
As $(\tau, n) = 1$ we obviously have $(\tau_{1}, n) = 1$ as well. We will prove that we also have $\xi^{\tau_{1} t - t'} = q^{m-\tau_{2}}$. First recall that, as noticed before, $\xi^{st' - t} = q^{l}$ implies $m \,|\, nl$. Now employing the equalities $\xi^{st' - t} = q^{l}$ and $s \tau + n \mu = 1$ we obtain
\begin{gather*}
\xi^{\tau_{1} t - t'} = \big(\xi^{t}\big)^{\tau_{1}} \xi ^{-t'} = \big(\xi^{st'} q^{-l}\big)^{\tau_{1}} \xi ^{-t'} =
\xi^{(s\tau_{1} -1)t'} q^{-l\tau_{1}} \stackrel{\eqref{impr}} {=} \xi^{(s\tau - s \alpha n -1)t'} q^{-l\tau_{1}} \\
\hphantom{\xi^{\tau_{1} t - t'}}{} = \xi^{(s\tau -1)t'} q^{-l\tau_{1}}= \xi^{-n \mu t'} q^{-l\tau_{1}} = q^{-l\tau_{1}} \stackrel{\eqref{impr2}} {=} q^{-\beta m + \alpha nl - \tau_{2}} \stackrel{m \,|\, nl} {=} q^{- \tau_{2}} = q^{m - \tau_{2}}.
\end{gather*}
Now since $\gamma^{-1} \in k^{*}$, $\tau_{1} \in \{0, 1,\dots, n-1\}$, $m - \tau_{2} \in \{0, 1,\dots, m-1\}$ and $\xi^{\tau_{1} t - t'} = q^{m-\tau_{2}}$, it follows from the first part of the proof that the map $\psi_{(u_{\gamma^{-1}}, \varepsilon, r_{m-\tau_{2}}, v_{\tau_{1}})}\colon T_{nm^{2}}^ {\xi^{t'}} \to T_{nm^{2}}^ {\xi^{t}}$ defined below is a Hopf algebra morphism
\begin{gather*}
\psi_{(u_{\gamma^{-1}}, \varepsilon, r_{m-\tau_{2}}, v_{\tau_{1}})}(a \# y) = u_{\gamma^{-1}}(a) r_{m-\tau_{2}}\big(y_{(1)}\big) \# v_{\tau_{1}}\big(y_{(2)}\big),
\end{gather*}
where $u_{\gamma^{-1}}\colon T_{m^{2}}(q) \to T_{m^{2}}(q)$, $r_{m-\tau_{2}}\colon K[C_{n}] \to T_{m^{2}}(q)$ and $v_{\tau_{1}}\colon K[C_{n}] \to K[C_{n}]$ are given as follows for all $i, j \in \{0, 1,\dots, m-1\}$ and $k \in \{0, 1,\dots, n-1\}$
\begin{gather*}
u_{\gamma^{-1}}\big(h^{i}x^{j}\big) = \gamma^{-j} h^{i}x^{j}, \qquad r_{m-\tau_{2}}\big(g^{k}\big) = h^{k(m-\tau_{2})}, \qquad v_{\tau_{1}}\big(g^{k}\big) = g^{k \tau_{1}}.
\end{gather*}
The proof will be finished once we show that $\psi_{(u_{\gamma^{-1}}, \varepsilon, r_{m-\tau_{2}}, v_{\tau_{1}})}$ is the
inverse of $\psi_{(u_{\gamma}, \varepsilon, r_{l}, v_{s})}$. Now for
all $i, j \in \{0, 1,\dots, m-1\}$ and $k \in \{0, 1,\dots, n-1\}$ we have
\begin{gather*}
\psi_{(u_{\gamma}, \varepsilon, r_{l}, v_{s})} \circ \psi_{(u_{\gamma^{-1}}, \varepsilon, r_{m-\tau_{2}}, v_{\tau_{1}})}\big(h^{i}x^{j} \# g^{k}\big) = \psi_{(u_{\gamma}, \varepsilon, r_{l}, v_{s})} \big(u_{\gamma^{-1}}\big(h^{i}x^{j}\big)r_{m-\tau_{2}}\big(g^{k}\big) \# v_{\tau_{1}}\big(g^{k}\big)\big) \\
\qquad \ {} = \psi_{(u_{\gamma}, \varepsilon, r_{l}, v_{s})} \big(\gamma^{-j} h^{i}x^{j}h^{k (m-\tau_{2})} \# g^{k \tau_{1}}\big)
 = \gamma^{-j} u_{\gamma}\big(h^{i}x^{j}h^{k (m-\tau_{2})}\big) r_{l}\big(g^{k \tau_{1}}\big) \# v_{s}\big(g^{k \tau_{1}}\big) \\
 \qquad \ {} = \gamma^{-j} \gamma^{j} h^{i}x^{j} h^{k (m-\tau_{2})} h^{k \tau_{1} l} \# g^{k \tau_{1} s}
 = h^{i}x^{j} h^{k (\underline{\tau_{1}} l -\underline{ \tau_{2}})} \# g^{k \underline{\tau_{1}} s} \\
\qquad{} \stackrel{\eqref{impr}} {=} h^{i}x^{j} h^{k(\tau l - \alpha n l - l \tau + \beta m)} \# g^{ks(\tau - \alpha n)}
 = h^{i}x^{j} h^{-k \alpha n l } \# g^{ks\tau} \\
 \qquad \ {} = h^{i}x^{j} h^{-k \alpha n l } \# g^{k(1-n \mu)} \stackrel{m \,|\, nl} = h^{i}x^{j} \# g^{k}.
\end{gather*}
Similarly one can prove that $\psi_{(u_{\gamma^{-1}}, \varepsilon, r_{m-\tau_{2}}, v_{\tau_{1}})} \circ \psi_{(u_{\gamma}, \varepsilon, r_{l}, v_{s})} (h^{i}x^{j} \# g^{k}) = h^{i}x^{j} \# g^{k}$ for all $i, j \in \{0, 1,\dots, m-1\}$ and $k \in \{0, 1,\dots, n-1\}$ and the proof is now finished.
\end{proof}

Using Theorem~\ref{izo} and Dirichlet's theorem \cite[Theorem~7.9]{ap} which states that for any positive integers $a$, $b$ such that ${\rm gcd}(a, b) = 1$ the set $\{a+kb \,|\, k \in \mathbb{N}\}$ contains an infinite number of primes, we will be able to provide an explicit formula for computing the number of isomorphism types of Hopf algebras which factorize through $T_{m^{2}}(q)$ and $K[C_n]$.

\begin{Theorem}\label{counting}
Let $\xi$ be a generator of $U_n (K)$ and $d = {\rm gcd}(m,\nu (n))$.
\begin{itemize}\itemsep=0pt
\item[$1)$] For any $t \in \{0, 1,\dots, \nu(n) - 1\}$ there exists an isomorphism of Hopf algebras $T_{nm^{2}}^ {\xi^{t}} \simeq T_{nm^{2}}^ {\xi^{{\rm gcd}\big(t, \frac{\nu(n)}{d}\big)}}$.
\item[$2)$] If $d_{1}$, $d_{2} \,|\, \frac{\nu(n)}{d}$, $d_{1} \neq d_{2}$, then the Hopf algebras $T_{nm^{2}}^ {\xi^{d_{1}}}$ and $T_{nm^{2}}^
{\xi^{d_{2}}}$ are not isomorphic.
\item[$3)$] If $ \frac{\nu(n)}{d} = p_1^{\alpha_1} \cdots p_r^{\alpha_r}$ is the prime decomposition of $ \frac{\nu(n)}{d}$ then the number of types of Hopf algebras that factorize through $T_{m^{2}}(q)$ and $K[C_n]$ is equal to $(\alpha_1 + 1)(\alpha_2 + 1) \cdots (\alpha_r + 1)$.
\end{itemize}
\end{Theorem}
\begin{proof} 1) Let $u = {\rm gcd}\left(t, \frac{\nu(n)}{d}\right)$. Thus we can write $t = au$, $ \frac{\nu(n)}{d} = bu$ with ${\rm gcd}(a, b) = 1$. By Dirichlet's theorem, the set $\{a+kb \,|\, k \in \mathbb{N}\}$ contains an infinite number of primes. In particular, there exists a prime number $s' \in \{a+kb \,|\, k \in \mathbb{N}\}$ such that ${\rm gcd}(s', n) = 1$ (otherwise, we would have $s' \,|\, n$ for an infinite number of primes). In fact, we can assume without loss of generality that there exists $s \in \{0, 1,\dots, n-1\}$ such that $(s, n) = 1$. Indeed, if $s' \notin \{0, 1,\dots, n-1\}$ then we can write $s' = s + nk_{1}$ for some $k_{1} \in \mathbb{N}$, $s \in \{0, 1,\dots, n-1\}$. Now ${\rm gcd}(s', n) = 1$ yields ${\rm gcd}(s, n) = 1$ as well. Since $s' \in \{a+kb \,|\, k \in \mathbb{N}\}$ we have $s' = a + k_{2}b$ for some $k_{2} \in \mathbb{N}$. Therefore, $s' = a + k_{2}b = s + nk_{1}$ which implies $b \,|\, s-a$ and by multiplying with~$u$ we obtain $ \frac{\nu(n)}{d} \,|\, su-t$, i.e., $\nu(n) \,|\, d(su-t)$. If $\nu(n) \,|\, su-t$ then $\xi^{su-t}=1$ and the conclusion follows from Theorem~\ref{izo} by considering $l=0$. Assume now that $\nu(n) \nmid su-t$ and $\nu(n) \,|\, d(su-t)$; in particular this implies that $d \geq 2$. Then we have $\xi^{d(su-t)} = 1$ which can be written equivalently as follows
\begin{gather*}
\big(\xi^{su-t} - 1\big)\big[\big(\xi^{su-t}\big)^{d-1} + \big(\xi^{su-t}\big)^{d-2} + \cdots + \xi^{su-t} +1\big] =0.
\end{gather*}
However, $\xi ^{su-t} \neq 1$ for otherwise we would have $\nu(n) \,|\, su-t$ which contradicts our assumption. Therefore, we have $(\xi^{su-t})^{d-1} + (\xi^{su-t})^{d-2} + \cdots + \xi^{su-t} +1 = 0$. Hence, $\xi^{su-t}$ is a solution of the equation $X^{d-1} + X^{d-2} + \cdots + X + 1 = 0$. The field $K$ contains all roots of unity of order $m$ as a consequence of containing the primitive root $q$ and therefore the equation $Y^{m} = 1$ has $m$ solutions, namely $1, q, q^{2},\dots, q^{m-1}$. Now since in particular we have $d \,|\, m$, the solutions of the equation $X^{d-1} + X^{d-2} + \cdots + X + 1 = 0$ are contained in the set $\{1, q, q^{2}, \dots, q^{m-1}\}$. Since we proved that $\xi^{su-t}$ is such a~solution, there exists some $l \in \{ 0,\dots, m-1 \}$ such that $\xi^{su-t} = q^{l}$. The desired conclusion now follows from Theorem~\ref{izo}.

2) Since $d_{1} \neq d_{2}$ there exists a prime number $p$ and two positive integers $\alpha, \beta \in \mathbb{N}$, $\alpha > \beta$, such that
\begin{gather*}
p^{\alpha} \,|\, d_{1}, \qquad p^{\beta} \,|\, d_{2}, \qquad p^{\alpha+1} \nmid d_{1}, \qquad p^{\beta+1} \nmid d_{2}.
\end{gather*}
Suppose now that the Hopf algebras $T_{nm^{2}}^ {\xi^{d_{1}}}$ and $T_{nm^{2}}^ {\xi^{d_{2}}}$ are isomorphic. By Theorem~\ref{izo} there
exists $l \in \{ 0, 1,\dots, m-1 \}$, $s \in \{ 0, 1, \dots, n-1 \}$ such that $(s, n) = 1$ and $\xi^{sd_{1} - d_{2}} = q^{l}$. In particular we obtain that $\nu(n) \,|\, m(sd_{1} - d_{2})$, i.e., $m(sd_{1} - d_{2}) = t \nu(n)$ for some $t \in \mathbb{Z}$. This comes down to the following
\begin{gather}\label{1.1.1}
 \frac{m}{d} \frac{d_{2}}{p^{\beta}} = \frac{m}{d} \frac{d_{1}}{p^{\beta}} s - \frac{\nu(n)}{dp^{\beta}} t.
\end{gather}
Since $\alpha > \beta$ we obtain $p \,|\, \frac{d_{1}}{p^{\beta}}$. Moreover, as $ d_{1} \,|\, \frac{\nu(n)}{d}$ we also have $p \,|\,\frac{d_{1}}{p^{\beta}} \,|\, \frac{\nu(n)}{dp^{\beta}}$ and therefore $p$ divides the right hand side of~\eqref{1.1.1}. Now $p \nmid
 \frac{m}{d}$. Indeed, if $p \,|\, \frac{m}{d}$, using $p \,|\, d_{1} \,|\, \frac{\nu(n)}{d}$ we would obtain $p \,|\, {\rm gcd}\big( \frac{m}{d}, \frac{\nu(n)}{d}\big) = 1$ which is a~contradiction. As we also have $p \nmid \frac{d_{2}}{q^{\beta}}$ we can conclude that $p
\nmid \frac{m}{d} \frac{d_{2}}{p^{\beta}}$. Therefore, $p$ does not divide the left hand side of~\eqref{1.1.1} and we have reached a~contradiction. Hence the Hopf algebras $T_{nm^{2}}^ {\xi^{d_{1}}}$ and $T_{nm^{2}}^ {\xi^{d_{2}}}$ are not isomorphic.

3) We already proved in 1) that any Hopf algebra $T_{nm^{2}}^{\xi^{t}}$ is isomorphic to $T_{nm^{2}}^ {\xi^{t'}}$, for some divisor~$t'$ of $ \frac{\nu(n)}{d}$. Moreover, by~2), if $d_1$ and $d_2$ are two distinct positive divisors of $ \frac{\nu(n)}{d}$ then $T_{nm^{2}}^ {\xi^{d_{1}}}$
and~$T_{nm^{2}}^ {\xi^{d_{2}}}$ are not isomorphic. Therefore, the set of types of Hopf algebras that factorize through~$T_{m^{2}}(q)$ and~$K[C_n]$ is in bijection with the set of all Hopf algebras~$T_{nm^{2}}^ {\xi^{t}}$, where $t$ is running over all positive divisors of~$ \frac{\nu(n)}{d}$.
\end{proof}

We end our investigation on the Hopf algebras which factorize through $T_{m^{2}}(q)$ and $K[C_n]$ by describing their automorphism group. To this end, for any $t \in \{0, 1,\dots, \nu - 1\}$ we define the set
\begin{gather*}
S^{t}_{m, n}(q) = \big\{\big(\overline{l}, \widehat{s}\big) \in (\mathbb{Z}_{m}, +) \times (U(\mathbb{Z}_{n}), \cdot) \,|\, \xi^{t(s-1)} = q^{l}\big\}.
\end{gather*}
$S^{t}_{m, n}(q)$ is actually a group with multiplication given for any $(\overline{l}, \widehat{s}), (\overline{l'},\widehat{s'}) \in S^{t}_{m, n}(q)$ as follows
\begin{gather*}
\big(\overline{l}, \widehat{s}\big)\big(\overline{l'}, \widehat{s'}\big) = \big(\overline{l+sl'}, \widehat{ss'}\big).
\end{gather*}
Indeed, $(\overline{0}, \widehat{1}) \in S^{t}_{m, n}(q)$ is the unit of the group while the inverse of some $(\overline{l}, \widehat{s}) \in S^{t}_{m, n}(q)$ is given by $(\overline{l}, \widehat{s})^{-1} = (\overline{-ls'}, \widehat{s'})$ where $\widehat{s'} =\widehat{s}^{-1}$ in $(U(\mathbb{Z}_{n}), \cdot)$. We will only prove that if $(\overline{l}, \widehat{s})$, $(\overline{l'},\widehat{s'}) \in S^{t}_{m, n}(q)$ then $(\overline{l+sl'}, \widehat{ss'}) \in S^{t}_{m, n}(q)$ as well. Indeed, as $(\overline{l}, \widehat{s})$, $(\overline{l'}, \widehat{s'})
\in S^{t}_{m, n}(q)$ we have $\xi^{t(s-1)} = q^{l}$ and $\xi^{t(s'-1)} = q^{l'}$. This gives $\xi^{t(s+s'-2)} = q^{l+l'}$. Moreover, we also have $\xi^{t(s'-1)(s-1)} = q^{l'(s-1)}$ which comes down to $\xi^{t(ss'- s - s' + 1)} = q^{l'(s-1)}$. By multiplying these two compatibilities yields $\xi^{t(ss'-1)} = q^{l+sl'}$ and therefore $(\overline{l+sl'}, \widehat{ss'}) \in S^{t}_{m, n}(q)$ as desired.

\begin{Theorem}\label{graut} For any $t \in \{0, 1,\dots, \nu - 1\}$ there exists an isomorphism of groups
\begin{gather*}
{\rm Aut} _{\rm Hopf}\big(T_{nm^{2}}^ {\xi^{t}}\big) \simeq K^{*}\times S^{t}_{m, n}(q).
\end{gather*}
\end{Theorem}
\begin{proof} According to (the proof of) Theorem~\ref{izo}, any Hopf algebra automorphism of $T_{nm^{2}}^ {\xi^{t}}$ has the following form
\begin{gather*}
\psi_{(u_{\gamma}, \varepsilon, r_{l}, v_{s})}(a \# y) = u_{\gamma}(a) r_{l}\big(y_{(1)}\big) \# v_{s}\big(y_{(2)}\big),
\end{gather*}
where $\gamma \in K^{*}$ and for any $i, j \in \{0, 1,\dots, m-1\}$, $k \in \{0, 1,\dots, n-1\}$ we have $u_{\gamma}(h^{i}x^{j}) = \gamma^{j}h^{i}x^{j}$, $r_{l}(g^{k}) = h^{kl}$ and $v_{s}(g^{k}) = g^{ks}$.

The proof will be finished once we show that the following map is an isomorphism of groups
\begin{align*}
\Gamma\colon \ {\rm Aut} _{\rm Hopf}\big(T_{nm^{2}}^ {\xi^{t}}\big) & \to K^{*}\times S^{t}_{m, n}(q), \\
 \psi_{(u_{\gamma}, \varepsilon, r_{l}, v_{s})} & \to \big(\gamma, (\overline{l}, \widehat{s})\big).
\end{align*}
To this end, we only need to prove that $\psi_{(u_{\gamma}, \varepsilon, r_{l}, v_{s})} \circ \psi_{(u_{\gamma '}, \varepsilon, r_{l'}, v_{s'})} = \psi_{(\gamma \gamma ', \varepsilon, l + sl' , ss')}$. We start by showing that $m | (l' + ls') - (l + l's)$ for all $(\overline{l}, \widehat{s})$, $(\overline{l}', \widehat{s}') \in S^{t}_{m, n}(q)$. Indeed, we have $\xi^{t(s-1)} = q^{l}$ and $\xi^{t(s'-1)} = q^{l'}$ which imply that $\xi^{t(s-1)(s'-1)} = q^{l(s'-1)} = q^{l'(s-1)}$. Thus $m | l(s'-1) - l'(s-1)$ and therefore we obtain
\begin{gather}\label{autom}
m | (l' + ls') - (l + l's).
\end{gather}
Now for any $i, j \in \{0, 1,\dots, m-1\}$ and $k \in \{0, 1,\dots, n-1\}$ we have
\begin{gather*}
\psi_{(u_{\gamma}, \varepsilon, r_{l}, v_{s})} \circ \psi_{(u_{\gamma '}, \varepsilon, r_{l'}, v_{s'})}\big(h^{i}
x^{j} \# g^{k}\big) = \psi_{(u_{\gamma}, \varepsilon, r_{l}, v_{s})}\big(u_{\gamma '}\big(h^{i}
x^{j}\big) r_{l'}\big(g^{k}\big) \# v_{s'}\big(g^{k}\big) \big) \\
\qquad \ {} =\psi_{(u_{\gamma}, \varepsilon, r_{l}, v_{s})}\big((\gamma')^{j} h^{i} x^{j}
h^{kl'} \# g^{ks'}\big) = (\gamma')^{j} q^{jkl'} \psi_{(u_{\gamma}, \varepsilon, r_{l}, v_{s})}
\big(h^{i+kl'} x^{j} \# g^{ks'}\big) \\
\qquad \ {} = (\gamma')^{j} q^{jkl'} u_{\gamma}\big(h^{i+kl'} x^{j}\big) r_{l}\big(g^{ks'}\big)
 \# v_{s}\big(g^{ks'}\big) =(\gamma')^{j} q^{jkl'} \gamma^{j} h^{i+kl'} x^{j} h^{kls'} \# g^{kss'} \nonumber\\
\qquad \ {} =(\gamma \gamma')^{j} q^{j(kl' + kls')} h^{i + kl' + kls'} x^{j} \# g^{kss'}
 =(\gamma \gamma')^{j} q^{kj(\underline{l' + ls'})} h^{i + k(\underline{l' + ls'})} x^{j} \# g^{kss'} \\
 \qquad {} \stackrel{\eqref{autom}} {=} (\gamma \gamma')^{j} q^{kj(l +
l's)} h^{i + k(l + l's)} x^{j} \# g^{kss'}
 = (\gamma \gamma')^{j} h^{i} x^{j} h^{k(l + l's)} \# g^{kss'} \\
 \qquad \ {} = \psi_{(\gamma \gamma ', \varepsilon, l + sl', ss')}\big(h^{i} x^{j} \# g^{k}\big)
\end{gather*}
as desired and the proof is now complete.
\end{proof}

\subsection*{Acknowledgements}

Parts of this work were undertaken while the author was visiting the Institut des Hautes \'{E}tudes Scientifiques in Bures-sur-Yvette, France. Their hospitality and the financial support offered by the Jean-Paul Gimon Chair are gratefully acknowledged. Also, we thank the referees for their comments and suggestions that substantially improved the first version of this paper.

\pdfbookmark[1]{References}{ref}
\LastPageEnding

\end{document}